\documentclass[final]{siamltex}
\usepackage{animate}
\usepackage{amsmath}
\usepackage{amssymb}
\usepackage{graphics}
\usepackage{subfigure}
\usepackage{graphicx}
\usepackage{textcomp}
\usepackage{mathrsfs}
\usepackage{epstopdf}
\usepackage{array}
\usepackage{cite}
\usepackage[maxfloats=99]{morefloats}
\usepackage{color}
\usepackage{url}
\usepackage{cases}
\usepackage{mathscinet}

\renewcommand{\theequation}{\arabic{section}.\arabic{equation}}

\usepackage[T1]{fontenc}
\usepackage{textcomp}
\usepackage[scaled]{helvet}

\newtheorem{remark}{Remark}[section]

\title{Superconvergence analysis of partially penalized immersed finite element method}
\author{Hailong Guo\thanks{ Department of Mathematics, University of California, Santa Barbara, CA 93106 (hlguo@math.ucsb.edu).}
    \and Xu Yang\thanks{Department of Mathematics, University of California, Santa Barbara, CA 93106 (xuyang@math.ucsb.edu). This work was partially supported by the NSF grant DMS-1418936 and Hellman Family Foundation Faculty Fellowship, UC Santa Barbara.}
    \and Zhimin Zhang\thanks{Beijing Computational Science Research Center, Beijing 100193 (zmzhang@csrc.ac.cn) and Department of Mathematics, Wayne State University, Detroit, MI 48202 (zzhang@math.wayne.edu). The research of this author was supported in part by the National Natural Science Foundation of China under grants 11471031, 91430216, U1530401, and the U.S. National Science Foundation through grant DMS-1419040.}
}

\begin{document}

\maketitle

%
%
\medskip

\begin{abstract}
The contribution of this paper contains two parts: first, we prove a supercloseness result for the partially penalized immersed finite element (PPIFE) method in [T. Lin, Y. Lin, and X. Zhang, SIAM J. Numer. Anal., 53 (2015), 1121--1144]; then based on the supercloseness result, we show that the gradient recovery method proposed in our previous work [H. Guo and X. Yang, arXiv: 1608.00063] can be applied to the PPIFE method and the recovered gradient converges to the exact gradient with a superconvergent rate $\mathcal{O}(h^{3/2})$.  Hence,
the gradient recovery method provides an asymptotically exact {\it a posteriori} error estimator for the PPIFE method.  Several numerical examples
are presented to verify our theoretical result.
\vskip .3cm

\end{abstract}


\section{Introduction}
Recently there has been of great interest in developing finite element method for interface problems where the discontinuous coefficients
appear naturally due to the background consisting of rather different materials; see, e.g.,\cite{BastianEngwer2009,Babuska1970,BarrettElliott1987,BrambleKing1996,CaiZhang2009,ChenDai2002,ChenZou1998,GongLiLi2007,Hansbo2002,HouLiu2005,HouWuZhang2004,Li1998,LiLinLin2004,LiLinWu2003,LiIto2006,LinLinZhang2015,Xu1982}.   It is well known that classical finite element method work
for interface problems provided that the mesh is aligned with the interface \cite{Babuska1970, BrambleKing1996,ChenZou1998,Xu1982}.
Such requirement may be a heavy burden especially when the interface involves complex geometry 
 and therefore it is  difficult and time-consuming to generate a body-fitted mesh.
To release the restriction,  Li proposed an immersed finite element (IFE) method for  the two-point boundary value problem \cite{Li1998}.
This idea was further generalized into  two-dimensional cases by Li, Lin, and Wu who constructed a nonconforming IFE method for
interface problems \cite{LiLinWu2003}.   
The main idea  of IFE is to solve interface problems on the Cartesian mesh (or uniform mesh) by  modifying basis functions near the interface.

The optimal approximation capability of IFE space was justified in \cite{LiLinLin2004}.  
However, there is no proof  for the optimal convergence  of the classical IFE method in the two-dimensional setting,  
see \cite{LinLinZhang2015}, even though plentiful numerical experiments showed optimal convergence for elliptic equations.  Interested readers are referred to \cite{ChouwKwakWee2010,HeLinLin2012,KwakWeeChang2010} for the progress of theoretical results.  
Moreover, numerical test results demonstrated that classic IFE method  
\cite{JiChenLi2014} achieves only the first order convergence in  the $L^{\infty}$-norm.
 There is relatively larger point-wise error over interfaces due to the discontinuities of test functions.  
To eliminate this disadvantage, the authors of \cite{JiChenLi2014} added a correction term into the bilinear form 
of the classic IFE method and  \cite{HouLiu2005,HouSongWangZhao2013,HouWuZhang2004} proposed a new IFE formulation  in the framework of the Petrov-Galerkin method.
However,  the  theoretical foundation of their methods is  not fully established.  Alternatively, Lin, Lin, and Zhang \cite{LinLinZhang2015} proposed PPIFE method to penalize the inter-element discontinuity. Thanks to the added penalty term, the authors proved the coercivity   of the bilinear form and showed the optimal convergence in the energy norm.

Superconvergence is an active research topic in the finite element community and its theory for smooth problems is well established,
see, e.g., \cite{Babuska2001, BankXu2003, Chen2001, ChenHuang1995, ChenXu2007, Lakhany2000, GuoZhang2015, Guo2016b, LinYan1996, NagaZhang2004, NagaZhang2005, Wahlbin1995, WuZhang2007, XuZhang2004, ZhangNaga2005, ZZ1992a, ZZ1992b, ZhuLin1989}, and
references therein. On the other hand, however, the superconvergence phenomena for interface problems
is not yet well understood  due to discontinuing of the coefficient crossing the interface.
  In \cite{WeiChenHuangZheng2014},  a supercloseness result between the linear finite element solution and its linear interpolation is proved for a two dimensional interface problem with body-fitted mesh.  Recently, the first two authors proposed  an immersed polynomial preserving recovery (IPPR) for interface problem and proved the superconvergence on both mildly unstructured mesh and adaptively refined mesh \cite{GuoYang2016}. For IFE method, Chou et al. introduced two special
interpolation formula to recover flux more accurately  for the one-dimensional linear and quadratic IFE elements
\cite{Chou2012,Chou2015}. In \cite{CaoZhangZhang2015}, Cao et al. investigated nodal superconvergence phenomena
using generalized orthogonal polynomial  in the one-dimensional setting.
For the two-dimensional case, the first two authors  proposed a new gradient recovery technique \cite{GuoYang2016b} for symmetric and consistent IFE method \cite{JiChenLi2014}
and Petrov-Galerlin  IFE method \cite{HouLiu2005,HouSongWangZhao2013,HouWuZhang2004} and numerically verified its superconvergence.  In addition, \cite{GuoYang2016b}
numerically showed  supercloseness results for both symmetric and consistent IFE method  and  Petrov-Galerlin  IFE method.

The main goal of this work is to establish a complete superconvergence theory for  the PPIFE method \cite{LinLinZhang2015}.  
Our analysis relies on the following three key observations: 
1) the solution is piecewise smooth on each sub-domain despite of its low global regularity;
2)  the basis functions on non-interface elements are just basis functions for standard linear finite element  method; 
3) the number of interface elements is  roughly $\mathcal{O}(h^{-1})$.
The above three observations motivate us to divide elements into disjoint types:  
interior elements, exterior elements, and  interface elements. 
We can obtain the supercloseness using  well-known results in \cite{BankXu2003, ChenXu2007, XuZhang2004} 
on interior and exterior elements, respectively.
In addition, the trace inequalities for the IFE functions in \cite{LinLinZhang2015} and the third observation enable us to establish 
$\mathcal{O}(h^{1.5})$ supercloseness result for interface elements. 
Our supercloseness result reduces to the standard one as in \cite{BankXu2003, ChenXu2007, XuZhang2004} when
the discontinuity disappears. 
It is consistent with the fact that IFE method becomes the standard linear finite element method when the discontinuity disappears.
Furthermore, we show that the gradient recovery method in \cite{GuoYang2016b} can also be applied to the PPIFE method.  
The recovered gradient is proven to be superconvergent to the exact gradient of the interface problem, and therefore,  
 provides  an asymptotically exact {\it a posterior} error estimator for PPIFE method.

The rest of the paper is organized as follows. In Section 2,  we introduce  the model interface problem and the PPIFE method.
In Section 3,  we first establish the supercloseness between gradients of the PPIFE solution and the exact solution to the interface problem, and then   based the supercloseness,  we  prove the recovered gradient using the method in \cite{GuoYang2016b} is superconvergent
to the exact gradient. Then provides an asymptotically exact {\it a posteriori} error estimator for PPIFE method.  
In Section 4, we present some numerical  experiments to support our theoretical result. 
Finally, we make some conclusive remarks in Section 5.

\section{Preliminary}
In this section, we shall introduce the elliptic interface problem, and  its its discrete form using the PPIFE method\cite{LinLinZhang2015}.
\subsection{Elliptic interface problem}
Let $\Omega$ be a bounded polygonal domain with  Lipschitz boundary $\partial \Omega$ in $\mathbb{R}^2$.   
A $C^2$-curve $\Gamma$ divides $\Omega$ into two disjoint subdomains $\Omega^-$ and $\Omega^+$, 
which is typically characterized by zero level set of some level set function $\phi$ \cite{Osher2003, Sethian1996},
with $\Omega^- = \{z\in \Omega|\phi(z) <0\}$ and $\Omega^+ = \{z\in \Omega|\phi(z) >0\}$.
We shall consider the following elliptic interface problem
\begin{align}
  -\nabla \cdot (\beta(z) \nabla u(z)) &= f(z),  \quad z \text{ in } \Omega\setminus \Gamma, \label{eq:model}\\
   u & = 0, \quad\quad\,\, z \text{ on } \partial\Omega, \label{eq:bnd}
\end{align}
where the diffusion coefficient $\beta(z) \ge \beta_0$ is a piecewise smooth function, i.e.
\begin{equation}
\beta(z) =
\left\{
\begin{array}{ccc}
    \beta^-(z) &  \text{if } z\in \Omega^-, \\
   \beta^+(z)  &   \text{if } z\in \Omega^+,
\end{array}
\right.
\end{equation}
which has a finite jump of function values across the interface $\Gamma$. At the interface $\Gamma$,
one has the following jump conditions
\begin{align}
   [u]_{\Gamma} &= u^+-u^-=0, \label{eq:valuejump}\\
   [\beta u_n]_{\Gamma} &= \beta^+u_n^+ - \beta^-u^-_n = g, \label{eq:fluxjump}
\end{align}
where $u_n$ denotes the normal flux $\nabla u\cdot n$ with $n$ as the unit outer normal vector of the interface $\Gamma$.

In this paper, we use  the standard
notations for Sobolev spaces and their associate norms given in \cite{BrennerScott2008, Ciarlet2002, Evans2008}.
For a subdomain $A$
of $\Omega$, let $\mathbb{P}_m(A)$ be the space of polynomials of
degree less than or equal to $m$ in $A$ and $n_m$ be the
dimension of $\mathbb{P}_m(A)$ which equals to $\frac{1}{2}(m+1)(m+2)$.
$W^{k,p}(A)$ denotes the Sobolev space with norm
$\|\cdot\|_{k, p, A} $ and seminorm $|\cdot|_{k, p,A}$.
 When $p = 2$, $W^{k,2}(A)$ is simply  denoted by $H^{k}(A)$
 and the subscript $p$ is omitted in its associate norm and seminorm.
  As in \cite{WeiChenHuangZheng2014}, denote
 $W^{k,p}(\Omega^-\cup\Omega^+)$  as the function space consisting of piecewise Sobolev  function $w$  such
 that $w|_{\Omega^-}\in W^{k,p}(\Omega^-)$ and $w|_{\Omega^+}\in W^{k,p}(\Omega^+)$.  For the function
 space $W^{k,p}(\Omega^-\cup\Omega^+)$, define norm  as
 \begin{equation*}
\|w\|_{k,p, \Omega^-\cup\Omega^+} = \left( \|w\|_{k,p, \Omega^-}^p + \|w\|_{k,p, \Omega^+}\right)^{1/p},
\end{equation*}
and seminorm as
 \begin{equation*}
|w|_{k,p, \Omega^-\cup\Omega^+} = \left( |w|_{k,p, \Omega^-}^p + |w|_{k,p, \Omega^+}\right)^{1/p}.
\end{equation*}
%

  We  assume that $\mathcal{T}_h$ is a shape regular triangulation of $\Omega$
with $h =\max\limits_{T\in \mathcal{T}} \mbox{diam}(T)_h$, and that $h$ is small enough 
so that the interface $\Gamma$ never crosses any edge of $\mathcal{T}_h$ more than once.
The elements of $\mathcal{T}_h$ can be divided into  two categories: regular elements and interface elements.
We call an element $T$ interface element if  the interface $\Gamma$ passes the interior of $T$; otherwise
we call it regular element. If $\Gamma$ only passes two vertices of an element $T$, we treat
the element $T$ as a regular element. Let $\mathcal{T}^{i}_h$ and $\mathcal{T}^{r}_h$ denote
the set of all interface elements and regular elements, respectively.

Let $\mathcal{N}_h$ and $\mathring{\mathcal{E}}_h$ denote the set of all vertices and interior edges of $\mathcal{T}_h$, respectively.
  We can divide   $\mathring{\mathcal{E}}_ h$ into two categories: interface edge $\mathring{\mathcal{E}}_h^{i}$ and regular edge
$\mathring{\mathcal{E}}_h^{r}$, which are  defined by
\begin{equation}
\mathring{\mathcal{E}}_h^{i} = \{e \in \mathring{\mathcal{E}}_h:  \mathring{e}\cap \Gamma \neq \emptyset\},
\,, \mathring{\mathcal{E}}_h^{r} = \mathring{\mathcal{E}}_h \setminus \mathring{\mathcal{E}}_h^{i}.
\end{equation}
For any interior edge $e$,  there exist two triangles $T_{e,1}$ and $T_{e,2}$ such that $T_{e,1}\cap T_{e,2}= e$.
Denote $n_e$ as the unit normal of $e$ pointing from $T_1$ to $T_2$, and define
\begin{align}
 &\left\{ u\right\} = \frac{1}{2}\left(  u|_{T_{e,1}} + u|_{T_{e,2}}\right),\label{eq:edgeaverage}\\
 &[u] = u|_{T_{e,1}} - u|_{T_{e,2}}. \label{eq:edgejump}
\end{align}
 When no confusion arises the subscript $e$  can be dropped.
We also introduce two  special function spaces $X_h$ and $X_{h,0}$ as
\begin{align}
&X_h := \left\{v\in X_h:  v|_{T} \in H^1(T) \text{ ,  }  v \text{ is continuous at  } \mathcal{N}_h \text{ and  across } \mathring{\mathcal{E}}_h^{r} \right\},\\
&X_{h,0} = \left\{v\in X_h:   v(z) = 0  \text{ for all } z \in \mathcal{N}_h\cap \partial\Omega \right\}.
\end{align}
We define a bilinear form $a_h: X_{h,0}\times X_{h,0} \rightarrow \mathbb{R}$ as
\begin{equation}\label{eq:bilinear}
\begin{split}
 a_h(v, w) = &\sum_{T\in \mathcal{T}_h}\int_{T}\beta\nabla v\cdot \nabla w dx - \sum_{e\in\mathring{\mathcal{E}}_h^{i} }
 \int_e\left\{ \beta\nabla v\cdot n_e\right\}[w]ds +\\
&\epsilon \sum_{e\in\mathring{\mathcal{E}}_h^{i} }
 \int_e\left\{ \beta\nabla w\cdot n_e\right\}[v]ds
 + \sum_{e\in\mathring{\mathcal{E}}_h^{i} }
 \int_e\frac{\sigma_e^0}{|e|}[v][w]ds,
\end{split}
\end{equation}
where the parameter $\sigma_e^0$ is positive and the parameter $\epsilon$ can be arbitrary.  Usually, $\epsilon$ takes the
value $-1$, $0$, or $1$. It is easy to see that $a_h$ is symmetric if $\epsilon = -1$ and it is nonsymmetric otherwise.

The general variational form \cite{LinLinZhang2015} of  \eqref{eq:model}-- \eqref{eq:fluxjump} is to find $u_h \in X_{h,0}$ such that
\begin{equation}\label{eq:var}
a_h(u, v) = (f, v), \quad \forall v \in X_{h,0}.
\end{equation}

\subsection{Partially penalized immersed finite element method}


\begin{figure}[ht]
    \centering
    \includegraphics[width=0.5\textwidth]{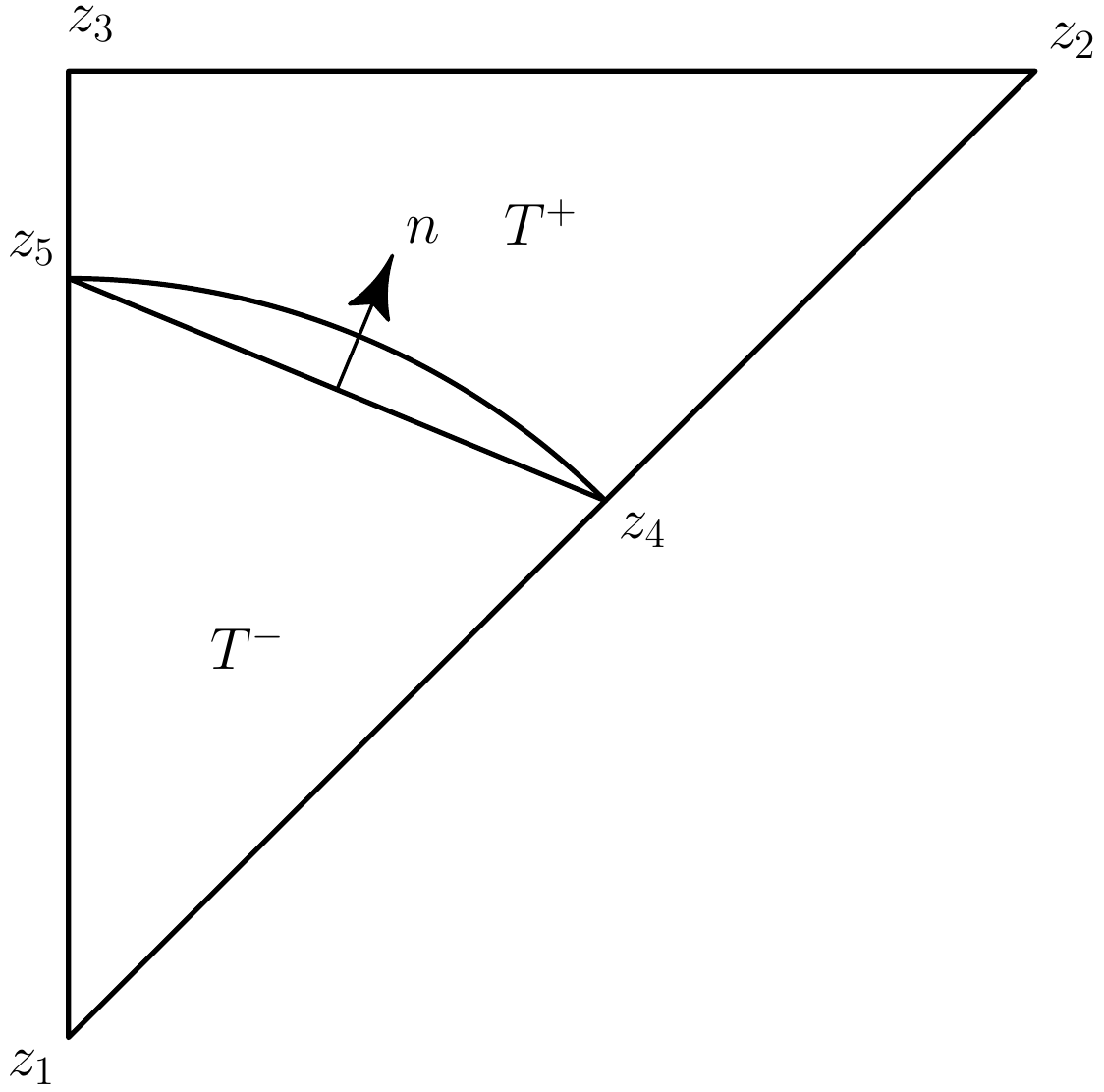}
    \caption{ Typical example of interface element.}
\label{fig:intelem}
\end{figure}

The key idea of partially penalized immersed finite element (PPIFE) method \cite{LinLinZhang2015} is to
modify  basis functions in interface elements to satisfy jump conditions \eqref{eq:valuejump} and \eqref{eq:fluxjump}.
Consider a typical interface element $T$ as in Figure \ref{fig:intelem}, and
let $z_4$ and $z_5$ be the intersection points between the interface $\Gamma$ and edges of the element.   Connecting the line segment
$\overline{z_4z_5}$  forms an approximation of interface $\Gamma$ in the element $T$, denoted by $\Gamma_h|_{T}$.
Then the element $T$ is split into two parts:  $T^-$ and $T^+$. We construct the following
 piecewise linear function on the interface element $T$
\begin{equation}\label{eq:ifembasis}
\phi(z)=
\left\{
\begin{array}{ccc}
\phi^+  = a^++b^+x+c^+y, & z = (x, y)\in T^+,    \\
 \phi^-= a^-+b^-x+c^-y, & z = (x, y)\in T^-,      \\
\end{array}
\right.
\end{equation}
where the coefficients are determined by the following linear system
\begin{align}
 &\phi(z_1) = V_1  , \, \phi(z_2) = V_2,\, \phi(z_3) = V_3,\label{eq:valueeq}\\
& \phi^+(z_4) = \phi^-(z_4), \, \phi^+(z_5) = \phi^-(z_5), \,
\beta^+\partial_n\phi^+ = \beta^-\partial_n\phi^-, \label{eq:fluxeq}
\end{align}
with $V_i$ being the nodal variables.  The immersed finite element space $V_h$  \cite{LiLinWu2003} is defined as
\begin{align}
&V_h := \left\{v\in V_h:  v|_{T} \in V_h(T) \text{ and }  v \text{ is continuous on } \mathcal{N}_h, \right\},\\
&V_{h,0} = \left\{v\in V_h:   v(z) = 0  \text{ for all } z \in \mathcal{N}_h\cap \partial\Omega \right\},
\end{align}
where
\begin{equation}
V_h(T) :=
\left\{
\begin{array}{ll}
 \left\{v| v \in \mathbb{P}_1(T) \right\},&   \text{if } T \in \mathcal{T}^{r}_h;    \\
  \left\{v| v \text{ is defined by } \eqref{eq:ifembasis}-\eqref{eq:fluxeq}\right\},&   \text{if } T \in \mathcal{T}^{i}_h .   \\
\end{array}
\right.
\end{equation}

A function in $V_h(T)$ is called a linear IFE function on $T$ when $T$ is an interface element.  For the linear IFE function, traditional
trace inequality \cite{BrennerScott2008, Ciarlet2002} fails.  In \cite{LinLinZhang2015}, Lin et al. established the following trace inequality:

\begin{lemma}
 There exists a constant $C$ independent of the interface location such that for every linear IFE function $v$ on $T$, the following
 inequality holds:
 \begin{equation}\label{eq:trace}
\|\beta \nabla \cdot n_e\|_{0, e} \le C h^{1/2}|T|^{-1/2}\|\sqrt{\beta}\nabla v\|_{0, K}.
\end{equation}

\end{lemma}

It is obvious that   $V_h$ (resp.  $V_{h,0}$) is a subspace of $X_h$ (resp.  $X_{h,0}$).   
The PPIFE
method for \eqref{eq:model}-- \eqref{eq:fluxjump} reads as finding $u_h \in V_{h,0}$ such that
\begin{equation}\label{eq:ppifem}
a_h(u_h, v_h) = (f, v_h), \quad \forall v_h \in V_{h,0}.
\end{equation}

 The energy norm $\|\cdot\|_h$ is defined as
\begin{equation}
\|v_h\|_h = \left( \sum_{T\in \mathcal{T}_h} \int_{T} \beta \nabla v_h \cdot \nabla v_h  dx + \sum_{e\in\mathring{\mathcal{E}}_h^{i} }
 \int_e\frac{\sigma_e^0}{|e|}[v_h]^2ds\right)^{1/2}.
\end{equation}
The following coercivity has been proved in \cite{LinLinZhang2015}:
\begin{lemma}\label{lem:coercivity}
There exists a constant $C>0$ such that
\begin{equation}
C\|v_h\|_h^2 \le a_h(v_h, v_h), \quad \forall v_h \in V_{h,0},
\end{equation}
is true for $\epsilon = 1$ unconditionally and is true for $\epsilon = 0$ or $\epsilon = -1$ under the condition that $\sigma^0_e$ is large enough.
\end{lemma}

Based on the above coercivity,  Lin  et al. proved the following optimal convergence result:
\begin{theorem}
 Assume that the exact solution $u$ to the interface problem \eqref{eq:model}-- \eqref{eq:fluxjump}   is in $H^3(\Omega^-\cup\Omega^+)$
 and $u_h$ is the solution to \eqref{eq:ppifem} on a Cartesian mesh $\mathcal{T}_h$. Then there exists a constant $C$ such that
 \begin{equation}
 \|u-u_h\|_h \le Ch \|u\|_{3, \Omega^-\cup\Omega^+}.
\end{equation}
\end{theorem}
\vspace{-0.15in}
\begin{remark}
 As remarked in \cite{LinLinZhang2015},  when the exact solution belongs to $W^{2, \infty}(\Omega^-\cup\Omega^+)$, the IFE solution $u_h$
 of \eqref{eq:ppifem}  on a Cartesian mesh $\mathcal{T}_h$ has error estimation in the following form
 \begin{equation}
 \|u-u_h\|_h \le C\left( h\|u\|_{3, \Omega^-\cup\Omega^+} + h^{3/2} \|u\|_{3, \infty, \Omega^-\cup\Omega^+}\right).
\end{equation}
Note that the above error estimation is also an optimal one since the leading (first) term is of $\mathcal{O}(h)$.
\end{remark}

\section{Superconvergence Analysis} In this section, we first present a superconvergence analysis 
for the PPIFE method on shape regular meshes.
Then we show  that the gradient recovery method introduced in \cite{GuoYang2016b} is applicable 
and prove that the recovered gradient is superconvergent to the exact gradient.
\subsection{Supercloseness result}  From now on,  we suppose $\mathcal{T}_h$ is a shape regular triangular mesh although $\mathcal{T}_h$ is usually
Cartesian mesh in the literature of IFE methods.  Let $h = \max\limits_{T\in \mathcal{T}_h} \text{diam}(T)$.  The set of regular element  $\mathcal{T}_h^r$ can be further decomposed into the
following two disjoint parts:
 \begin{equation}\label{eq:meshdecom}
\begin{split}
 &\mathcal{T}^-_h:=\left\{  T\in \mathcal{T}_h^r|  T \text{ has all three vertices in } \overline{\Omega^-}     \right\},\\
 & \mathcal{T}^+_h:=\left\{  T\in \mathcal{T}_h^r|  T  \text{ has all three vertices in }\overline{\Omega^+}  \right\} .\\
\end{split}
\end{equation}

\begin{definition}
1. Two adjacent triangles are called to  form an $\mathcal{O}(h^{1+\alpha})$  approximate parallelogram if the lengths of any two opposite edges differ only by $\mathcal{O}(h^{1+\alpha})$.

2. The triangulation $\mathcal{T}_h$ is called to satisfy Condition
$(\sigma,\alpha)$ if
there exist a partition $\mathcal{T}_{h,1} \cup \mathcal{T}_{h,2}$ of $\mathcal{T}_h$
and positive constants $\alpha$ and $\sigma$ such that every
two adjacent triangles in $\mathcal{T}_{h,1}$ form an $\mathcal{O}(h^{1+\alpha})$ parallelogram and
$$
\sum_{T\in {\mathcal{T}_{h,2}}} |T| = \mathcal{O}(h^\sigma).
$$
\end{definition}
\vspace{-0.15in}
\begin{remark}
 It is obvious that Cartesian mesh satisfies  Condition $(\sigma,\alpha)$ with $\sigma = \infty$ and $\alpha = 1$.
\end{remark}

Suppose $\mathcal{T}_h$ satisfies  Condition $(\sigma,\alpha)$. Then we can prove the following
supercloseness result:
\begin{theorem}\label{thm:supercloseness}
 Suppose the triangulation $\mathcal{T}_h$ satisfies  Condition $(\sigma,\alpha)$.  Let $u$ be the solution of
 the interface problem  \eqref{eq:model}-- \eqref{eq:fluxjump} and $u_I$ be the interpolation of $u$ in the IFE space
 $V_{h,0}$. If $u \in H^1(\Omega) \cap H^{3}(\Omega^-\cup\Omega^+)\cap W^{2, \infty}(\Omega^-\cup\Omega^+)$,
 then for all $v_h \in V_{h,0}$
 \begin{equation}\label{eq:superclose}
a_h(u-u_I, v_h) \le C\left(h^{1+\rho}(\|u\|_{3, \Omega^+\cup\Omega^-} + \|u\|_{2, \infty, \Omega^+\cup\Omega^-}) + Ch^{3/2} \|u\|_{2, \infty, \Omega^+\cup\Omega^-}\right)|v_h|_{h}.
\end{equation}
where $C$ is a constant independent of interface location and $h$ and $\rho = \min(\alpha, \frac{\sigma}{2}, \frac{1}{2})$.
\end{theorem}
\begin{proof}
 Notice that
 \begin{equation}
\begin{split}
 & a_h(u-u_I, v_h)\\ = &\sum_{T\in \mathcal{T}_h}\int_{T}\beta\nabla (u-u_I)\cdot \nabla v_h dx  -
  \sum_{e\in\mathring{\mathcal{E}}_h^{i} }
 \int_e\left\{ \beta\nabla (u-u_I)\cdot n_e\right\}[v_h]ds +\\
&\epsilon \sum_{e\in\mathring{\mathcal{E}}_h^{i} }
 \int_e\left\{ \beta\nabla v_h\cdot n_e\right\}[u-u_I]ds  +
 \sum_{e\in\mathring{\mathcal{E}}_h^{i} }
 \int_e\frac{\sigma_e^0}{|e|}[u-u_I][v_h]ds \\
=&\sum_{T\in \mathcal{T}_h^+}\int_{T}\beta\nabla (u-u_I)\cdot \nabla v_h dx +
\sum_{T\in \mathcal{T}_h^-}\int_{T}\beta\nabla (u-u_I)\cdot \nabla v_h dx \\
&\sum_{T\in \mathcal{T}_h^i}\int_{T}\beta\nabla (u-u_I)\cdot \nabla v_h dx  -
  \sum_{e\in\mathring{\mathcal{E}}_h^{i} }
 \int_e\left\{ \beta\nabla (u-u_I)\cdot n_e\right\}[v_h]ds +\\
&\epsilon \sum_{e\in\mathring{\mathcal{E}}_h^{i} }
 \int_e\left\{ \beta\nabla v_h\cdot n_e\right\}[u-u_I]ds  +
 \sum_{e\in\mathring{\mathcal{E}}_h^{i} }
 \int_e\frac{\sigma_e^0}{|e|}[u-u_I][v_h]ds\\
 = &I_1 + I_2 + I_3 + I_4 + I_5 + I_6.
\end{split}
\end{equation}
Since $\mathcal{T}_h$ satisfies Condition $(\sigma,\alpha)$, it follows that $\mathcal{T}_h^+$ and $\mathcal{T}_h^-$  also satisfy Condition $(\sigma,\alpha)$ .
Using the fact that the IFE functions becoming standard linear function on regular element, we have the following estimates for $I_1$ and $I_2$, whose proof can be found in
\cite{XuZhang2004}:
\begin{align}
 |I_1|\le Ch^{1+\rho}(\|u\|_{3, \Omega^+} + \|u\|_{2, \infty, \Omega^+})|v_h|_{h},\\
  |I_2|\le Ch^{1+\rho}(\|u\|_{3, \Omega^-} + \|u\|_{2, \infty, \Omega^-})|v_h|_{h},
\end{align}
where $C$ is a constant independent of $h$ and $\rho = \min(\alpha, \frac{\sigma}{2}, \frac{1}{2})$.  Now we proceed to estimate $I_3$.
By the Cauchy-Schwartz inequality, we have
\begin{equation}
\begin{split}
I_3 = &\sum_{T\in \mathcal{T}_h^i}\int_{T}\beta\nabla (u-u_I)\cdot \nabla v_h dx \\
 \le& \left( \sum_{T\in \mathcal{T}_h^i} \|\beta^{1/2}\nabla (u-u_I)\|^2_{0, T}\right)^{1/2}
 \left( \sum_{T\in \mathcal{T}_h^i} \|\beta^{1/2}\nabla v_h\|^2_{0, T}\right)^{1/2}\\
 \le& \left( \sum_{T\in \mathcal{T}_h^i} \max(\beta^-, \beta^+) \|\nabla (u-u_I)\|^2_{0, T}\right)^{1/2}
 \left( \sum_{T\in \mathcal{T}_h^i}\|\beta^{1/2}\nabla v_h\|^2_{0, T}\right)^{1/2}\\
  \le&C\left( \sum_{T\in \mathcal{T}_h^i}  h^2\|u\|_{2, T^-\cup T^+}^2\right)^{1/2}
 \left( \sum_{T\in \mathcal{T}_h^i}\|\beta^{1/2}\nabla v_h\|^2_{0, T}\right)^{1/2}\\
   \le&C\left( \sum_{T\in \mathcal{T}_h^i}  h^4\|u\|_{2, \infty, T^-\cup T^+}^2\right)^{1/2}
 \left( \sum_{T\in \mathcal{T}_h^i}\|\beta^{1/2}\nabla v_h\|^2_{0, T}\right)^{1/2}\\
    \le&Ch^2\|u\|_{2, \infty, \Omega^-\cup \Omega^+}\left( \sum_{T\in \mathcal{T}_h^i} 1  \right)^{1/2}
 \left( \sum_{T\in \mathcal{T}_h^i}\|\beta^{1/2}\nabla v_h\|^2_{0, T}\right)^{1/2}\\
  \le&Ch^{3/2}\|u\|_{2, \infty, \Omega^-\cup \Omega^+}
 \left( \sum_{T\in \mathcal{T}_h^i}\|\beta^{1/2}\nabla v_h\|^2_{0, T}\right)^{1/2}\\
   \le&Ch^{3/2}\|u\|_{2, \infty, \Omega^-\cup \Omega^+}\|v_h\|_h,
\end{split}
\end{equation}
where we have used optimal approximation capability of IFE space \cite{LiLinLin2004, LiLinWu2003} and the fact that $ \sum_{T\in \mathcal{T}_h^i} 1 \approx \mathcal{O}(h^{-1})$.
Then we estimate $I_4$.  Cauchy-Schwartz inequality implies that
\begin{equation}
\begin{split}
 I_4 = &\sum_{e\in\mathring{\mathcal{E}}_h^{i} }\int_e\left\{ \beta\nabla (u-u_I)\cdot n_e\right\}[v_h]ds \\
 \le& \left( \sum_{e\in\mathring{\mathcal{E}}_h^{i} }\int_e \frac{|e|}{\sigma_e^0}\left\{ \beta\nabla (u-u_I)\cdot n_e\right\}^2 ds\right)^{1/2}
   \left( \sum_{e\in\mathring{\mathcal{E}}_h^{i} }\int_e\frac{\sigma_e^0}{|e|} [v_h]^2 ds\right)^{1/2}\\
    \le& Ch^{1/2}\left( \sum_{e\in\mathring{\mathcal{E}}_h^{i} }\int_e \left\{ \beta\nabla (u-u_I)\cdot n_e\right\}^2 ds\right)^{1/2}
   \left( \sum_{e\in\mathring{\mathcal{E}}_h^{i} }\int_e\frac{\sigma_e^0}{|e|} [v_h]^2 ds\right)^{1/2}\\
    \le& Ch^{2}\|u\|_{2, \infty, \Omega^-\cup\Omega^+}\left( \sum_{T\in \mathcal{T}_h^i} 1\right)^{1/2}
   \left( \sum_{e\in\mathring{\mathcal{E}}_h^{i} }\int_e\frac{\sigma_e^0}{|e|} [v_h]^2 ds\right)^{1/2}\\
      \le& Ch^{3/2}\|u\|_{2, \infty, \Omega^-\cup\Omega^+}\|v_h\|_h,\\
\end{split}
\end{equation}
where we have used (4.19) in \cite{LinLinZhang2015}.  To bound $I_5$, we use the standard  trace inequality \cite{BrennerScott2008, Ciarlet2002} which implies
\begin{equation}\label{eq:edgeapprox}
\begin{split}
 \|[u-u_I]\|_{0, e} \le& \| (u-u_I)_{T_{e,1}}\|_{0,e} + \|(u-u_I)_{T_{e,2}}\|_{0,e}\\
 \le & Ch^{-1/2}\left(\|u-u_I\|_{0, T_{e,1}} + h\|\nabla(u-u_I)\|_{0,T_{e,1}}\right) + \\
 &Ch^{-1/2}\left(\|u-u_I\|_{0, T_{e,2}} + h\|\nabla(u-u_I)\|_{0,T_{e,2}}\right) \\
  \le & Ch^{3/2}\left(\|u\|_{2,  T_{e,1}^-\cup  T_{e,1}^+} + \|u\|_{2,  T_{e,2}^-\cup  T_{e,1}^+}\right) \\
    \le & Ch^{5/2}\|u\|_{2,  \infty, \Omega^-\cup  \Omega^+} .
     \end{split}
\end{equation}
Also, the trace inequality for IFE function \eqref{eq:trace} implies that
\begin{equation}
\begin{split}
 \|\left\{ \beta\nabla v_h\cdot n_e\right\}\|_{0, e} \le& \ \|\left\{ \beta\nabla v_h|_{T_{e,1}}\cdot n_e\right\}\|_{0, e}  +  \|\left\{ \beta\nabla v_h|_{T_{e,2}}\cdot n_e\right\}\|_{0, e}\\
\le & Ch^{-1/2} \left( \|\sqrt{\beta} \nabla v_h\|_{0, T_{e,1}} +  \|\sqrt{\beta} \nabla v_h\|_{0, T_{e,2}}\right).
     \end{split}
\end{equation}
Hence, we get
\begin{equation}
\begin{split}
I_5 = &\left|\epsilon \sum_{e\in\mathring{\mathcal{E}}_h^{i} }
 \int_e\left\{ \beta\nabla v_h\cdot n_e\right\}[u-u_I]ds \right|\\
  \le& \left( \sum_{e\in\mathring{\mathcal{E}}_h^{i} }\|\left\{ \beta\nabla v_h\cdot n_e\right\}\|_{0,e}^2 \right)^{1/2}
   \left( \sum_{e\in\mathring{\mathcal{E}}_h^{i} }\|[u-u_I]\|_{0,e}^2 \right)^{1/2}\\
     \le&C \left( \sum_{e\in\mathring{\mathcal{E}}_h^{i} }h^{-1} \left( \|\sqrt{\beta} \nabla v_h\|_{0, T_{e,1}} +  \|\sqrt{\beta} \nabla v_h\|_{0, T_{e,2}}\right)^2 \right)^{1/2}  \left( \sum_{e\in\mathring{\mathcal{E}}_h^{i} }h^5\|u\|_{2,  \infty, \Omega^-\cup  \Omega^+}^2 \right)^{1/2}\\
\le&Ch^{2}\|u\|_{2,  \infty, \Omega^-\cup  \Omega^+} \left(  \sum_{T\in \mathcal{T}_h^i}   \|\sqrt{\beta} \nabla v_h\|_{0, T} ^2 \right)^{1/2}   \left(  \sum_{T\in \mathcal{T}_h^i} 1 \right)^{1/2}\\
\le &Ch^{3/2}\|u\|_{2,  \infty, \Omega^-\cup  \Omega^+}\|v_h\|_h,
 \end{split}
\end{equation}
where we have also used the fact $ \sum_{T\in \mathcal{T}_h^i} 1 \approx \mathcal{O}(h^{-1})$.
For $I_6$, by the Cauchy-Schwartz inequality and \eqref{eq:edgeapprox}, we have
\begin{equation}
\begin{split}
 I_6 = &\sum_{e\in\mathring{\mathcal{E}}_h^{i} }\int_e\frac{\sigma_e^0}{|e|} [u-u_I][v_h]ds \\
 \le& \left( \sum_{e\in\mathring{\mathcal{E}}_h^{i} }\int_e \frac{\sigma_e^0}{|e|}[ u-u_I]^2 ds\right)^{1/2}
   \left( \sum_{e\in\mathring{\mathcal{E}}_h^{i} }\int_e\frac{\sigma_e^0}{|e|} [v_h]^2 ds\right)^{1/2}\\
    \le&  Ch^{-1/2}\left( \sum_{e\in\mathring{\mathcal{E}}_h^{i} } \|[u-u_I]\|_{0,e}^2\right)^{1/2}\|v_h\|_h\\
    \le & Ch^{2}\|u\|_{2,  \infty, \Omega^-\cup  \Omega^+}\left( \sum_{T\in \mathcal{T}_h^i}  1\right)^{1/2}\|v_h\|_h\\
    \le &  Ch^{3/2}\|u\|_{2,  \infty, \Omega^-\cup  \Omega^+}\|v_h\|_h,
   \end{split}
\end{equation}
where we have also used the fact $ \sum_{T\in \mathcal{T}_h^i} 1 \approx \mathcal{O}(h^{-1})$.
Summarizing the bounds for $I_i$ ($i = 1, 2, \cdots, 6$) together gives \eqref{eq:superclose}.
\end{proof}

\begin{remark}\label{rmk:reduce}
 When the discontinuity disappears, $\mathring{\mathcal{E}}_h^{i}$ will become empty.  In that case,  $I_i$ (i = 3, 4, 5, 6) will become zero and we can
 reproduce  the standard supercloseness result  \cite{XuZhang2004}.
\end{remark}

\begin{remark}
 Here we discuss the triangle element.  For the bilinear PPIFE method,  we can prove  similar supercloseness results by adapting
 the integral identities in \cite{LinYanZhou1991,LinYan1996}, the trace inequalities for bilinear IFE functions \cite{LinLinZhang2015}, and the same techniques that we used here to deal with the interface part.
\end{remark}

Based on the supercloseness results, we can prove the following theorem:
\begin{theorem}\label{thm:superinterp}
 Assume the same hypothesis in Theorem \ref{thm:supercloseness} and let $u_h$ be the IFE solution of discrete variational problem
 \eqref{eq:ppifem} ; then
 \begin{equation}\label{eq:superconvergence}
\|u_h-u_I\|_h \le C\left(h^{1+\rho}(\|u\|_{3, \Omega^+\cup\Omega^-} + \|u\|_{2, \infty, \Omega^+\cup\Omega^-}) + Ch^{3/2} \|u\|_{2, \infty, \Omega^+\cup\Omega^-}\right),
\end{equation}
where $\rho = \min(\alpha, \frac{\sigma}{2}, \frac{1}{2})$.
\end{theorem}
\begin{proof}
 Since $V_{h,0}$ is a subset of $X_{h,0}$, it follows that
 \begin{equation}
a_h(u-u_h, v_h) = 0, \quad\forall v_h\in V_{h,0}.
\end{equation}
  Then we have
\begin{equation}
a_h(u_h-u_I, v_h)  = a_h(u-u_I, v_h), \quad \forall v_h\in V_{h,0}.
\end{equation}
Taking $v_h = u_h-u_I$ and using Theorem \ref{thm:supercloseness} and Lemma \ref{lem:coercivity}, we prove \eqref{eq:superconvergence}.
\end{proof}

\begin{remark}
 Similarly as Remark \ref{rmk:reduce}, when the discontinuity disappear,  \eqref{eq:superconvergence} will reduce to  the standard supercloseness result  \cite{XuZhang2004}.
\end{remark}

\subsection{Superconvergence results}  In this subsection, using the supercloseness results,
 we show that the recovered gradient  of the PPIFE solution is superconvergent to the exact gradient.

To define the gradient recovery operator introduced \cite{GuoYang2016b}, we first generate a local 
body-fitted mesh $\widehat{\mathcal{T} }_h$ by adding some new vertices into $\mathcal{N}_h$ \cite{GuoYang2016b, LiLinWu2003}.  
Suppose $ \widehat{X}_h$ is a $C^0$ linear finite element space defined on $\widehat{\mathcal{T} }_h$, we
defined an enrich operator  $E_h: V_h\rightarrow  \widehat{X}_h$ by averaging the discontinuous values on interface vertices as in \cite{GuoYang2016b}.

Let $\Gamma_h$ be the approximated interface by connecting the intersection points of edges with $\Gamma$.  
We can category  the triangulation $\widehat{\mathcal{T}}_h$ into the following two disjoint sets:
\begin{align}
 &\widehat{\mathcal{T}}^-_h:=\left\{  T\in \mathcal{T}_h| \text{ all three vertices of  } T \text{ are in  }\overline{\Omega^-}  \right\},\label{equ:mmesh}\\
 & \widehat{\mathcal{T}}^+_h:=\left\{  T\in \mathcal{T}_h| \text{ all three vertices of  } T \text{ are in  }\overline{\Omega^+}  \right\} . \label{equ:pmesh}
\end{align}
Let $\Omega^-_h = \cup_{T\in \widehat{\mathcal{T}}_h^-} T$ and $\Omega^+_h = \cup_{T\in \widehat{\mathcal{T}}_h^+} T$.
Suppose $\widehat{X}_h^-$ and $\widehat{X}_h^+$ are the continuous linear finite element spaces defined on $ \widehat{\mathcal{T}}^-_h$ and $ \widehat{\mathcal{T}}^+_h$, respectively.

Denote the PPR gradient recovery operator on $\widehat{X}_h^-$ by $G_h^-$.  Then $G_h^-$ is a linear bounded operator from $\widehat{X}_h^-$ to $\widehat{X}_h^-\times \widehat{X}_h^-$.
Similarly, let $G_h^+$ be PPR gradient recovery operator from $\widehat{X}_h^+$ to $\widehat{X}_h^+\times \widehat{X}_h^+$. Then, for any $u_h \in V_h$, let $G_h^I:  \widehat{X}_h \rightarrow ( \widehat{X}_h^-\cup  \widehat{X}_h^+) \times ( \widehat{X}_h^-\cup  \widehat{X}_h^+)$ be the immersed polynomial preserving recovery (IPPR) operator
proposed in \cite{GuoYang2016} which is defined as  following. 
\begin{equation}
(G_h^I u_h) (z)=
\left\{
\begin{array}{ccc}
   (G_h^- u_h) (z)  &  \text{if } z\in \overline{\Omega^-_h}, \\
   (G_h^+ u_h) (z)  &   \text{if } z\in \overline{\Omega^+_h}.
\end{array}
\right.
\end{equation}
  Then the recovered gradient of PPIFE solution $u_h$  is defined as
\begin{equation}\label{equ:rg}
R_hu_h = G_h^I(E_hv_h).
\end{equation}

The linear boundedness and consistency of the gradient recovery operator $R_h$ are showed in \cite{GuoYang2016b}.
The previous established supercloseness result enables us to prove the following main superconvergence result:
\begin{theorem}\label{thm:super}
 Assume the same hypothesis in Theorem \ref{thm:supercloseness} and let $u_h$ be the IFE solution of discrete variational problem
 \eqref{eq:ppifem} ; then
 \begin{equation}\label{equ:super}
\|\nabla u-R_hu_h\|_{0,\Omega} \le C\left(h^{1+\rho}(\|u\|_{3, \Omega^+\cup\Omega^-} + \|u\|_{2, \infty, \Omega^+\cup\Omega^-}) + Ch^{3/2} \|u\|_{2, \infty, \Omega^+\cup\Omega^-}\right).
\end{equation}
\end{theorem}
\begin{proof}
 We decompose $\nabla u-R_hu_h$ as  $ (\nabla u - R_hu_I) -  (R_hu_I - R_hu_h)$.   Then the triangle inequality implies that
 \begin{equation}\label{equ:pi}
 \|\nabla u-R_hu_h\|_{0,\Omega} \le \|\nabla u - R_hu_I\|_{0,\Omega} +  \| R_hu_I - R_hu_h\|_{0,\Omega}:=I_1+I_2.
\end{equation}
According to Theorem 3.7 in \cite{GuoYang2016b}, we have
\begin{equation}\label{equ:pi1}
I_1 \lesssim h^2\|u\|_{3, \Omega^-\cup\Omega^+}.
\end{equation}
Using definition \eqref{equ:rg}, we obtain that
\begin{equation}\label{equ:pi2}
\begin{split}
 I_2 =&  \| G_hE_h(u_I - u_h)\|_{0,\Omega} \\
    \lesssim & \| G_h E_h(u_I-u_h)\|_{0,\Omega_h^-} + \| G_h E_h(u_I-u_h)\|_{0,\Omega_h^+}\\
    \lesssim & \|\nabla E_h(u_I-u_h)\|_{0,\Omega_h^-} +  \|\nabla  E_h(u_I-u_h)\|_{0,\Omega_h^+}\\
    \lesssim & \|\nabla E_h(u_I-u_h)\|_{0,\Omega}\\
    \lesssim & \|\nabla (u_I-u_h)\|_{0,\Omega}\\
    \lesssim & h^{1+\rho}(\|u\|_{3, \Omega^+\cup\Omega^-} + \|u\|_{2, \infty, \Omega^+\cup\Omega^-}) + Ch^{3/2} \|u\|_{2, \infty, \Omega^+\cup\Omega^-}
\end{split}
\end{equation}
where we have used the boundedness property of $G_h^{\pm}$ in the second inequality,  
Corollary 3.4 of \cite{GuoYang2016b} in the fourth inequality, and Theorem \ref{thm:superinterp} in the last inequality.  
Combining \eqref{equ:pi}-\eqref{equ:pi2} completes the proof of  \eqref{equ:super}.
\end{proof}

The gradient recovery operator $R_h$ naturally provides an {\it a posteriori } error estimators for the  PPIFE method.
 We define a local {\it a posteriori} error estimator on element $T\in \mathcal{T}_h$ as
\begin{equation}\label{equ:localind}
\eta_T =
\left\{
\begin{array}{lcc}
    \|\beta^{1/2}(R_hu_h - \nabla u_h)\|_{0, T}, &  \text{if } T \in \mathcal{T}_h^r, \\
   \left(\sum\limits_{\widehat{T}\subset T, \widehat{T}\in \widehat{T}_h} \|\beta^{1/2}(R_hu_h - \nabla u_h)\|_{0, \widehat{T}}^2\right)^{\frac{1}{2}}, &   \text{if } T \in \mathcal{T}_h^i,
\end{array}
\right.
\end{equation}
and the corresponding global error estimator  as
\begin{equation}\label{equ:globalind}
\eta_h = \left( \sum_{T\in \mathcal{T}_h}\eta_T^2\right)^{1/2}.
\end{equation}

With the above superconvergence result, we are ready to prove the asymptotic exactness of error estimators based on the recovery operator $R_h$.
\begin{theorem}\label{thm:asyexact}
 Assume the same hypothesis in Theorem \ref{thm:supercloseness} and let $u_h$ be the IFE solution of discrete variational problem
 \eqref{eq:ppifem}.  Further assume that there is a constant $C(u)>0$ such that
 \begin{equation}\label{equ:satassum}
 \|\nabla (u-u_h)\|_{0,\Omega} \ge C(u) h.
\end{equation}
Then it holds that
\begin{equation}
\left |  \frac{\eta_h}{\|\nabla (u-u_h)\|_{0,\Omega}}  -1 \right | \lesssim h^{\rho}.
\end{equation}
\end{theorem}
\begin{proof}
 It follows  from  Theorem \ref{thm:super},  \eqref{equ:satassum}, and the triangle inequality.
\end{proof}
\begin{remark}
 Theorem \ref{thm:asyexact} implies that \eqref{equ:localind} (or \eqref{equ:globalind}) is an asymptotically exact {\it a posteriori } error estimator for PPIFE method.
\end{remark}

\section{Numerical Examples}
In the section, the previous established supercloseness and superconvergence theory are demonstrated by three numerical examples.
The first two are benchmark problems for testing numerical methods for linear interface problem.
For that two examples,  the computational domain  are chosen as $\Omega = [-1, 1]\times [-1, 1]$.
  The uniform triangulation of $\Omega$ is obtained by dividing $\Omega$ into $N^2$
sub-squares and then dividing each sub-square into two right triangles. The resulting uniform  mesh size is $h = 1/N$.
The last example is a nonlinear interface problem.
We test the examples using three different PPIFE method\cite{LinLinZhang2015}:
the symmetric PPIFE method (SPPIFEM),  incomplete  PPIFE method (IPPIFEM),  and
non-symmetric PPIFE method (NPPIFEM), which are corresponding to
$\epsilon = -1$, $\epsilon=0$, and $\epsilon=1$, respectively. We choose
the penalty parameter $\sigma_e^0=\sqrt{\max(\beta^-,\beta^+)}$ for SPPIFEM and IPPIFEM and
$\sigma_e^0=1$ for NPPIFEM.  For convenience, we shall adopt the following error norms in all the examples:
\begin{equation}
De:=\|u-u_h\|_{1,\Omega},\quad D^ie:=\|\nabla u_I- \nabla u_h\|_{0, \Omega},\quad
D^re:=\|\nabla u-R_hu_h\|_{0, \Omega}.\\
 \end{equation}

{\bf Example 4.1.} In this example, we  consider  the elliptic interface problem  \eqref{eq:model} with a circular interface of radius $r_0 = \frac{\pi}{6}$ as studied in \cite{LiLinWu2003}.
 The exact solution is
\begin{equation*}
u(z) = \left\{
\begin{array}{ll}
    \frac{r^3}{\beta^-}   &  \text{if }   z\in \Omega_-, \\
      \frac{r^3}{\beta^+} + \left( \frac{1}{\beta^-}-\frac{1}{\beta^+} \right)r_0^3&  \text{if } z \in \Omega^+,\\
   \end{array}
\right.
\end{equation*}
where $r = \sqrt{x^2+y^2}$.

We use two typical jump rations: $\beta^-/\beta^+=1/10$ and $\beta^-/\beta^+=1/1000$.
Tables \ref{tab:ex1aspp}-\ref{tab:ex1bnpp} report numerical results. For $De$, all three partially penalized
finite element methods converge  with the optimal rate $\mathcal{O}(h)$ for both differential jump ratios.
As for $D^ie$ and $D^re$,  $\mathcal{O}(h^{1.5})$ order of convergence can be clearly observed for all cases,
which support our Theorems 3.3 and 3.4.

\begin{table}[htb!]
\centering
\caption{SPPIFEM for Example 4.1 with $\beta^+=10, \beta^-=1$. }\label{tab:ex1aspp}
\begin{tabular}{|c|c|c|c|c|c|c|c|}
\hline
 $h$ & $De$ & order& $D^{i}e$ & order& $D^{r}_re$ & order\\ \hline\hline
 1/16 &7.20e-02&--&3.64e-03&--&1.91e-02&--\\ \hline
 1/32 &3.62e-02&0.99&1.34e-03&1.44&5.10e-03&1.91\\ \hline
 1/64 &1.81e-02&1.00&4.64e-04&1.53&1.68e-03&1.60\\ \hline
 1/128 &9.07e-03&1.00&1.58e-04&1.56&5.24e-04&1.68\\ \hline
 1/256 &4.53e-03&1.00&5.76e-05&1.45&1.71e-04&1.62\\ \hline
 1/512 &2.27e-03&1.00&2.00e-05&1.53&5.87e-05&1.54\\ \hline
 1/1024 &1.13e-03&1.00&7.07e-06&1.50&1.94e-05&1.60\\ \hline
\end{tabular}
\end{table}

\begin{table}[htb!]
\centering
\caption{IPPIFEM for Example 4.1 with $\beta^+=10, \beta^-=1$. }\label{tab:ex1aipp}
\begin{tabular}{|c|c|c|c|c|c|c|c|}
\hline
 $h$ & $De$ & order& $D^{i}e$ & order& $D^{r}_re$ & order\\ \hline\hline
 1/16 &7.20e-02&--&3.61e-03&--&1.90e-02&--\\ \hline
 1/32 &3.62e-02&0.99&1.22e-03&1.57&5.01e-03&1.92\\ \hline
 1/64 &1.81e-02&1.00&3.98e-04&1.62&1.63e-03&1.62\\ \hline
 1/128 &9.07e-03&1.00&1.38e-04&1.53&5.07e-04&1.68\\ \hline
 1/256 &4.53e-03&1.00&4.94e-05&1.48&1.64e-04&1.63\\ \hline
 1/512 &2.27e-03&1.00&1.72e-05&1.52&5.58e-05&1.55\\ \hline
 1/1024 &1.13e-03&1.00&6.05e-06&1.51&1.83e-05&1.61\\ \hline
\end{tabular}
\end{table}

\begin{table}[htb!]
\centering
\caption{NPPIFEM for Example 4.1 with $\beta^+=10, \beta^-=1$. }\label{tab:ex1anpp}
\begin{tabular}{|c|c|c|c|c|c|c|c|}
\hline
 $h$ & $De$ & order& $D^{i}e$ & order& $D^{r}_re$ & order\\ \hline\hline
 1/16 &7.20e-02&--&3.90e-03&--&1.89e-02&--\\ \hline
 1/32 &3.62e-02&0.99&1.29e-03&1.59&4.97e-03&1.93\\ \hline
 1/64 &1.81e-02&1.00&4.18e-04&1.63&1.59e-03&1.64\\ \hline
 1/128 &9.07e-03&1.00&1.44e-04&1.53&4.97e-04&1.68\\ \hline
 1/256 &4.53e-03&1.00&5.17e-05&1.48&1.60e-04&1.64\\ \hline
 1/512 &2.27e-03&1.00&1.80e-05&1.52&5.43e-05&1.56\\ \hline
 1/1024 &1.13e-03&1.00&6.34e-06&1.51&1.77e-05&1.61\\ \hline
\end{tabular}
\end{table}

\begin{table}[htb!]
\centering
\caption{SPPIFEM for Example 4.1 with $\beta^+=1000, \beta^-=1$. }\label{tab:ex1bspp}
\begin{tabular}{|c|c|c|c|c|c|c|c|}
\hline
 $h$ & $De$ & order& $D^{i}e$ & order& $D^{r}_re$ & order\\ \hline\hline
 1/16 &2.47e-02&--&4.60e-03&--&1.33e-02&--\\ \hline
 1/32 &1.31e-02&0.91&1.78e-03&1.37&3.62e-03&1.88\\ \hline
 1/64 &6.56e-03&1.00&6.44e-04&1.47&1.36e-03&1.42\\ \hline
 1/128 &3.31e-03&0.99&2.17e-04&1.57&4.60e-04&1.56\\ \hline
 1/256 &1.65e-03&1.01&7.70e-05&1.49&1.38e-04&1.73\\ \hline
 1/512 &8.23e-04&1.00&2.72e-05&1.50&4.71e-05&1.55\\ \hline
 1/1024 &4.12e-04&1.00&9.60e-06&1.50&1.59e-05&1.57\\ \hline
\end{tabular}
\end{table}

\begin{table}[htb!]
\centering
\caption{IPPIFEM for Example 4.1 with $\beta^+=1000, \beta^-=1$. }\label{tab:ex1bipp}
\begin{tabular}{|c|c|c|c|c|c|c|c|}
\hline
 $h$ & $De$ & order& $D^{i}e$ & order& $D^{r}_re$ & order\\ \hline\hline
 1/16 &2.54e-02&--&8.58e-03&--&1.50e-02&--\\ \hline
 1/32 &1.35e-02&0.91&3.86e-03&1.15&4.99e-03&1.58\\ \hline
 1/64 &6.65e-03&1.02&1.29e-03&1.58&1.79e-03&1.48\\ \hline
 1/128 &3.33e-03&1.00&4.36e-04&1.57&5.46e-04&1.72\\ \hline
 1/256 &1.65e-03&1.01&1.55e-04&1.50&1.83e-04&1.58\\ \hline
 1/512 &8.25e-04&1.00&5.60e-05&1.47&6.61e-05&1.46\\ \hline
 1/1024 &4.12e-04&1.00&2.01e-05&1.48&2.38e-05&1.48\\ \hline
\end{tabular}
\end{table}

\begin{table}[htb!]
\centering
\caption{NPPIFEM for Example 4.1 with $\beta^+=1000, \beta^-=1$. }\label{tab:ex1bnpp}
\begin{tabular}{|c|c|c|c|c|c|c|c|}
\hline
 $h$ & $De$ & order& $D^{i}e$ & order& $D^{r}_re$ & order\\ \hline\hline
 1/16 &2.56e-02&--&9.39e-03&--&1.55e-02&--\\ \hline
 1/32 &1.36e-02&0.91&4.29e-03&1.13&5.34e-03&1.54\\ \hline
 1/64 &6.67e-03&1.03&1.41e-03&1.61&1.88e-03&1.50\\ \hline
 1/128 &3.34e-03&1.00&4.84e-04&1.54&5.65e-04&1.74\\ \hline
 1/256 &1.65e-03&1.01&1.72e-04&1.49&1.95e-04&1.53\\ \hline
 1/512 &8.25e-04&1.00&6.30e-05&1.45&7.23e-05&1.43\\ \hline
 1/1024 &4.12e-04&1.00&2.29e-05&1.46&2.67e-05&1.44\\ \hline
\end{tabular}
\end{table}

{\bf Example 4.2.}  In this example, we consider the interface problem  \eqref{eq:model} with a cardioid interface as in \cite{HouLiu2005}.
The interface curve $\Gamma$ is  the zero level of the function
\begin{equation*}
\phi(x,y) = (3(x^2+y^2)-x)^2-x^2-y^2,
\end{equation*}
as shown Figure \ref{fig:ex42}.    We choose  the exact solution $u(x,y) = \phi(x,y)/\beta(x,y)$, where
\begin{equation*}
\beta(x,y) =
\left\{
\begin{array}{lcc}
    xy+3 &  \text{if } (x,y)\in \Omega^-, \\
   100 &   \text{if } (x,y)\in \Omega^+.
\end{array}
\right.
\end{equation*}

Note that   the interface  is not Lipschitz-continuous and has a singular point at the origin.
Tables \ref{tab:ex2spp}-\ref{tab:ex2npp} display the numerical data. We observe
the same supercloseness and superconvergence phenomena as predicted by our theory.

\begin{figure}[ht]
\centering
     \includegraphics[width=0.5\textwidth]{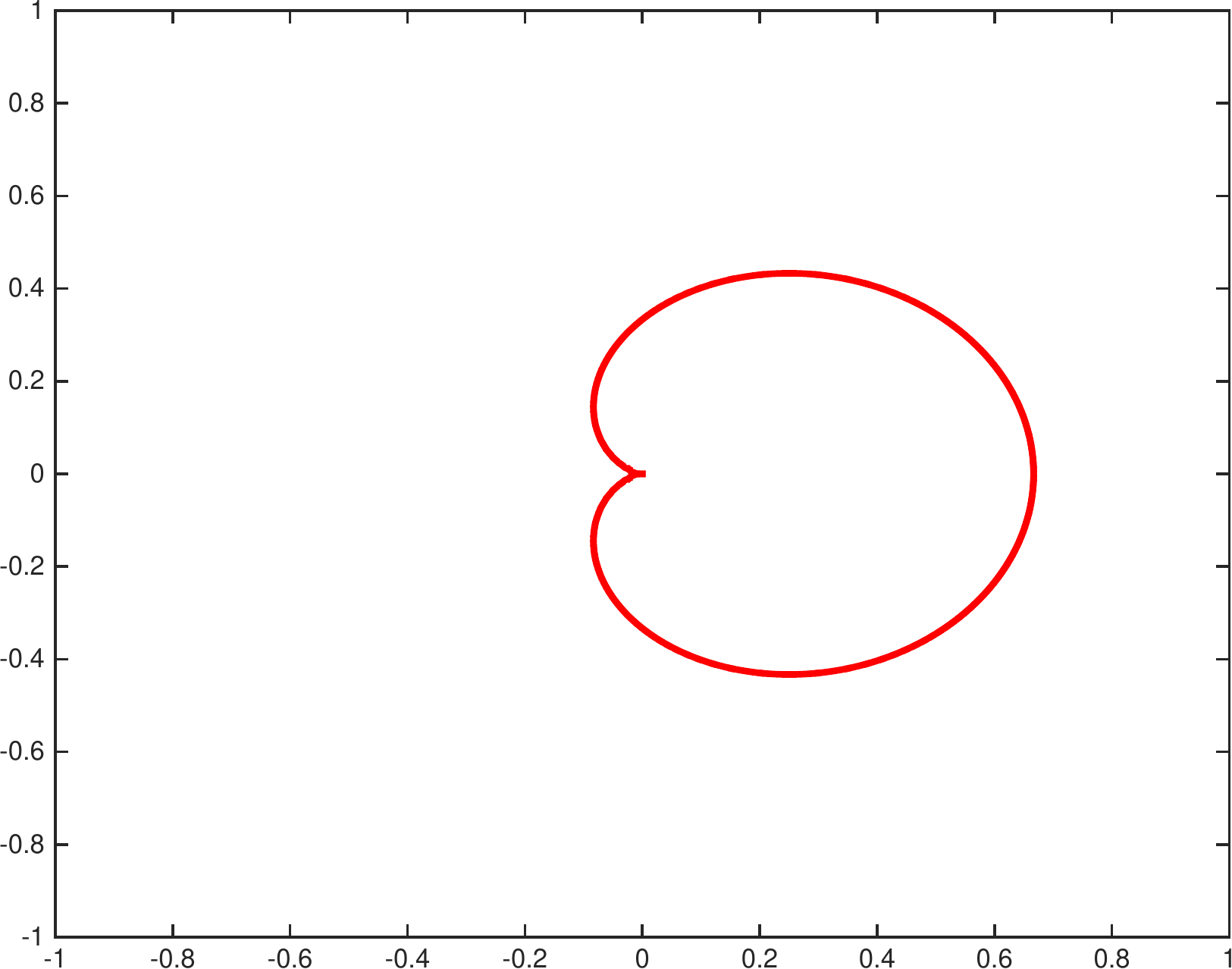}
\caption{Shape of interface  for Example 4.2}
\label{fig:ex42}
\end{figure}

\begin{table}[htb!]
\centering
\caption{SPPIFEM for Example 4.2 .}\label{tab:ex2spp}
\begin{tabular}{|c|c|c|c|c|c|c|c|}
\hline
 $h$ & $De$ & order& $D^{i}e$ & order& $D^{r}_re$ & order\\ \hline\hline
 1/16 &5.77e-02&--&5.71e-03&--&2.48e-02&--\\ \hline
 1/32 &3.03e-02&0.93&2.39e-03&1.26&7.88e-03&1.66\\ \hline
 1/64 &1.51e-02&1.01&8.55e-04&1.48&2.27e-03&1.80\\ \hline
 1/128 &7.43e-03&1.02&3.07e-04&1.48&7.10e-04&1.68\\ \hline
 1/256 &3.71e-03&1.00&1.13e-04&1.45&2.36e-04&1.59\\ \hline
 1/512 &1.86e-03&1.00&3.94e-05&1.52&8.84e-05&1.41\\ \hline
 1/1024 &9.32e-04&1.00&1.39e-05&1.50&3.08e-05&1.52\\ \hline
\end{tabular}
\end{table}

\begin{table}[htb!]
\centering
\caption{IPPIFEM for Example 4.2.}\label{tab:ex2ipp}
\begin{tabular}{|c|c|c|c|c|c|c|c|}
\hline
 $h$ & $De$ & order& $D^{i}e$ & order& $D^{r}_re$ & order\\ \hline\hline
 1/16 &5.77e-02&--&6.45e-03&--&2.46e-02&--\\ \hline
 1/32 &3.03e-02&0.93&2.67e-03&1.27&7.63e-03&1.69\\ \hline
 1/64 &1.51e-02&1.01&9.68e-04&1.46&2.25e-03&1.76\\ \hline
 1/128 &7.43e-03&1.02&3.55e-04&1.45&6.98e-04&1.69\\ \hline
 1/256 &3.71e-03&1.00&1.25e-04&1.51&2.30e-04&1.60\\ \hline
 1/512 &1.86e-03&1.00&4.39e-05&1.51&8.51e-05&1.43\\ \hline
 1/1024 &9.32e-04&1.00&1.54e-05&1.51&2.95e-05&1.53\\ \hline
\end{tabular}
\end{table}

\begin{table}[htb!]
\centering
\caption{NPPIFEM for Example 4.2.}\label{tab:ex2npp}
\begin{tabular}{|c|c|c|c|c|c|c|c|}
\hline
 $h$ & $De$ & order& $D^{i}e$ & order& $D^{r}_re$ & order\\ \hline\hline
 1/16 &5.78e-02&--&7.96e-03&--&2.47e-02&--\\ \hline
 1/32 &3.03e-02&0.93&3.12e-03&1.35&7.58e-03&1.70\\ \hline
 1/64 &1.51e-02&1.01&1.17e-03&1.41&2.29e-03&1.72\\ \hline
 1/128 &7.43e-03&1.02&4.35e-04&1.43&7.16e-04&1.68\\ \hline
 1/256 &3.71e-03&1.00&1.51e-04&1.53&2.35e-04&1.61\\ \hline
 1/512 &1.86e-03&1.00&5.31e-05&1.50&8.62e-05&1.45\\ \hline
 1/1024 &9.32e-04&1.00&1.87e-05&1.51&3.00e-05&1.52\\ \hline
\end{tabular}
\end{table}

{\bf Example 4.3.}  In this example, we consider the following nonlinear interface problem
\begin{equation*}
  -\nabla \cdot (\beta(z) \nabla u(z)) + \sin(u(z)) = f(z),  \quad z \text{ in } \Omega\setminus \Gamma,
\end{equation*}
with homogeneous jump conditions \eqref{eq:valuejump} and \eqref{eq:fluxjump} where $\Omega = [-2, 2]\times [-2, 2] \setminus [-0.5, 0.5]\times [-0.5, 0.5]$.
  The interface curve $\Gamma$ is circle centered at origin with radius $r_0=\pi/3$.
 The exact solution is
\begin{equation*}
u(z) =
\left\{
\begin{array}{ll}
    \frac{\log(r)}{\beta^-},   &  \text{if }  z\in \Omega_-, \\
      \frac{\log(r)}{\beta^+} + \left( \frac{1}{\beta^-}-\frac{1}{\beta^+} \right)\log(r_0),&  \text{if } z \in \Omega^+,\\
   \end{array}
\right.
\end{equation*}
where $r =|z|=\sqrt{x^2+y^2}$.  The right hand side function $f$ and boundary condition are obtained from the exact solution.

The nonlinear interface problem is solved by  the PPIFE method
with Newton's iteration on a series of uniform meshes.
The coarsest mesh is depicted  in
Fig \ref{fig:ex43} and the finer meshes are obtained by the uniform refinement.  
Numerical results are reported in Tables \ref{tab:ex3spp}-\ref{tab:ex3npp}.
 We observe the same superconvergence and supercloseness phenomena as linear problems.

\begin{figure}[ht]
\centering
     \includegraphics[width=0.5\textwidth]{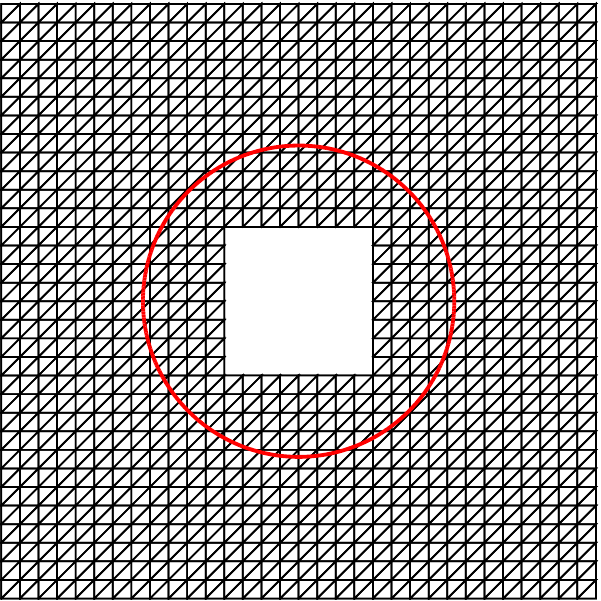}
\caption{Initial non body-fitted mesh for Example 4.3}
\label{fig:ex43}
\end{figure}

\begin{table}[htb!]
\centering
\caption{SPPIFEM for Example 4.3 with $\beta^+=1000, \beta^-=1$.}\label{tab:ex3spp}
\begin{tabular}{|c|c|c|c|c|c|c|c|}
\hline
 $h$ & $De$ & order& $D^{i}e$ & order& $D^{r}_re$ & order\\ \hline\hline
 1/8 &1.69e-01&--&2.55e-02&--&6.44e-02&--\\ \hline
 1/16 &8.53e-02&0.99&8.54e-03&1.58&1.60e-02&2.01\\ \hline
 1/32 &4.19e-02&1.03&3.03e-03&1.49&6.13e-03&1.38\\ \hline
 1/64 &2.09e-02&1.00&1.05e-03&1.53&2.17e-03&1.50\\ \hline
 1/128 &1.04e-02&1.01&3.70e-04&1.51&6.69e-04&1.70\\ \hline
 1/256 &5.17e-03&1.00&1.28e-04&1.54&2.34e-04&1.51\\ \hline
 1/512 &2.58e-03&1.00&4.46e-05&1.52&7.90e-05&1.57\\ \hline
\end{tabular}
\end{table}

\begin{table}[htb!]
\centering
\caption{IPPIFEM for Example 4.3 with $\beta^+=1000, \beta^-=1$.}\label{tab:ex3ipp}
\begin{tabular}{|c|c|c|c|c|c|c|c|}
\hline
 $h$ & $De$ & order& $D^{i}e$ & order& $D^{r}_re$ & order\\ \hline\hline
 1/8 &1.75e-01&--&5.15e-02&--&7.52e-02&--\\ \hline
 1/16 &8.71e-02&1.00&1.90e-02&1.44&2.32e-02&1.70\\ \hline
 1/32 &4.23e-02&1.04&6.15e-03&1.63&8.27e-03&1.49\\ \hline
 1/64 &2.10e-02&1.01&1.97e-03&1.64&2.49e-03&1.73\\ \hline
 1/128 &1.04e-02&1.01&6.85e-04&1.52&8.61e-04&1.53\\ \hline
 1/256 &5.17e-03&1.01&2.42e-04&1.50&3.17e-04&1.44\\ \hline
 1/512 &2.58e-03&1.00&8.77e-05&1.47&1.14e-04&1.47\\ \hline
\end{tabular}
\end{table}

\begin{table}[htb!]
\centering
\caption{NPPIFEM for Example 4.3 with $\beta^+=1000, \beta^-=1$..}\label{tab:ex3npp}
\begin{tabular}{|c|c|c|c|c|c|c|c|}
\hline
 $h$ & $De$ & order& $D^{i}e$ & order& $D^{r}_re$ & order\\ \hline\hline
1/8 &1.77e-01&--&6.03e-02&--&8.06e-02&--\\ \hline
 1/16 &8.78e-02&1.01&2.23e-02&1.44&2.55e-02&1.66\\ \hline
 1/32 &4.24e-02&1.05&6.89e-03&1.69&8.73e-03&1.55\\ \hline
 1/64 &2.10e-02&1.01&2.25e-03&1.62&2.63e-03&1.73\\ \hline
 1/128 &1.04e-02&1.02&7.84e-04&1.52&9.28e-04&1.50\\ \hline
 1/256 &5.17e-03&1.01&2.80e-04&1.49&3.47e-04&1.42\\ \hline
 1/512 &2.58e-03&1.00&1.02e-04&1.45&1.29e-04&1.43\\ \hline
\end{tabular}
\end{table}

\section{Conclusion}  In this paper, we study the superconvergence theory for partially penalized immersed finite element (PPIFE) method. Specifically, we obtain supercloseness results analogous to
standard linear finite element method. Due to the existence of the interface, we can only prove a supercloseness result of order $\mathcal{O}(h^{1.5})$.  We also notice that the supercloseness  result will reduce to the well known one for  standard linear element when the discontinuity disappears.  These results provide us a fundamental tool to
prove the $\mathcal{O}(h^{1.5})$  superconvergence of recovered gradient by using the gradient recovery operator proposed in \cite{GuoYang2016b}.    
We present three numerical examples to
support our theoretical results.  

%

\begin{thebibliography}{10}

\bibitem{BrennerScott2008}
{\sc Susanne~C. Brenner and L.~Ridgway Scott}, {\em The mathematical theory of
  finite element methods}, vol.~15 of Texts in Applied Mathematics, Springer,
  New York, third~ed., 2008.

\bibitem{Ciarlet2002}
{\sc Philippe~G. Ciarlet}, {\em The finite element method for elliptic
  problems}, vol.~40 of Classics in Applied Mathematics, Society for Industrial
  and Applied Mathematics (SIAM), Philadelphia, PA, 2002.
\newblock Reprint of the 1978 original [North-Holland, Amsterdam; MR0520174 (58
  \#25001)].

\bibitem{Evans2008}
{\sc Lawrence~C. Evans}, {\em Partial differential equations}, vol.~19 of
  Graduate Studies in Mathematics, American Mathematical Society, Providence,
  RI, second~ed., 2010.

\bibitem{GuoYang2016b}
{\sc Hailong Guo and Xu~Yang}, {\em Gradient recovery for elliptic interface
  problem: {II}. immersed finite element methods}, 2016.
\newblock arXiv:1608.000063 [math.NA].

\bibitem{LiLinLin2004}
{\sc Z.~Li, T.~Lin, Y.~Lin, and R.~C. Rogers}, {\em An immersed finite element
  space and its approximation capability}, Numer. Methods Partial Differential
  Equations, 20 (2004), pp.~338--367.

\bibitem{LiLinWu2003}
{\sc Zhilin Li, Tao Lin, and Xiaohui Wu}, {\em New {C}artesian grid methods for
  interface problems using the finite element formulation}, Numer. Math., 96
  (2003), pp.~61--98.

\bibitem{LinLinZhang2015}
{\sc Tao Lin, Yanping Lin, and Xu~Zhang}, {\em Partially penalized immersed
  finite element methods for elliptic interface problems}, SIAM J. Numer.
  Anal., 53 (2015), pp.~1121--1144.

\bibitem{Osher2003}
{\sc Stanley Osher and Ronald Fedkiw}, {\em Level set methods and dynamic
  implicit surfaces}, vol.~153 of Applied Mathematical Sciences,
  Springer-Verlag, New York, 2003.

\bibitem{Sethian1996}
{\sc J.~A. Sethian}, {\em Level set methods}, vol.~3 of Cambridge Monographs on
  Applied and Computational Mathematics, Cambridge University Press, Cambridge,
  1996.
\newblock Evolving interfaces in geometry, fluid mechanics, computer vision,
  and materials science.

\bibitem{WeiChenHuangZheng2014}
{\sc Huayi Wei, Long Chen, Yunqing Huang, and Bin Zheng}, {\em Adaptive mesh
  refinement and superconvergence for two-dimensional interface problems}, SIAM
  J. Sci. Comput., 36 (2014), pp.~A1478--A1499.

\bibitem{XuZhang2004}
{\sc Jinchao Xu and Zhimin Zhang}, {\em Analysis of recovery type a posteriori
  error estimators for mildly structured grids}, Math. Comp., 73 (2004),
  pp.~1139--1152 (electronic).

\end{thebibliography}


\begin{thebibliography}{10}

\bibitem{Babuska1970}
{\sc Ivo Babu{\v{s}}ka}, {\em The finite element method for elliptic equations
  with discontinuous coefficients}, Computing (Arch. Elektron. Rechnen), 5
  (1970), pp.~207--213.

\bibitem{Babuska2001}
{\sc Ivo Babu{\v{s}}ka and Theofanis Strouboulis}, {\em The finite element
  method and its reliability}, Numerical Mathematics and Scientific
  Computation, The Clarendon Press, Oxford University Press, New York, 2001.

\bibitem{BankXu2003}
{\sc Randolph~E. Bank and Jinchao Xu}, {\em Asymptotically exact a posteriori
  error estimators. {I}. {G}rids with superconvergence}, SIAM J. Numer. Anal.,
  41 (2003), pp.~2294--2312 (electronic).

\bibitem{BarrettElliott1987}
{\sc John~W. Barrett and Charles~M. Elliott}, {\em Fitted and unfitted
  finite-element methods for elliptic equations with smooth interfaces}, IMA J.
  Numer. Anal., 7 (1987), pp.~283--300.

\bibitem{BastianEngwer2009}
{\sc Peter Bastian and Christian Engwer}, {\em An unfitted finite element
  method using discontinuous {G}alerkin}, Internat. J. Numer. Methods Engrg.,
  79 (2009), pp.~1557--1576.

\bibitem{BrambleKing1996}
{\sc James~H. Bramble and J.~Thomas King}, {\em A finite element method for
  interface problems in domains with smooth boundaries and interfaces}, Adv.
  Comput. Math., 6 (1996), pp.~109--138 (1997).

\bibitem{BrennerScott2008}
{\sc Susanne~C. Brenner and L.~Ridgway Scott}, {\em The mathematical theory of
  finite element methods}, vol.~15 of Texts in Applied Mathematics, Springer,
  New York, third~ed., 2008.

\bibitem{CaiZhang2009}
{\sc Zhiqiang Cai and Shun Zhang}, {\em Recovery-based error estimator for
  interface problems: conforming linear elements}, SIAM J. Numer. Anal., 47
  (2009), pp.~2132--2156.

\bibitem{CaoZhangZhang2015}
{\sc Waixiang Cao, Xu~Zhang, and Zhimin Zhang}, {\em Superconvergence of
  immersed finite element methods for interface problems}, 2016.
\newblock arXiv:1511.04648 [math.NA].

\bibitem{Chen2001}
{\sc Chuanmiao Chen}, {\em Structure Theory of Superconvergence of Finite
  Elements (in Chinese)}, Hunan Science and Technique Press, Changsha, 2001.
  
\bibitem{ChenHuang1995}
{\sc Chuanmiao Chen and Yunqing Huang}, {\em High Accuracy Theory of Finite
  Element Methods (in Chinese)}, Hunan Science and Technique Press, Changsha, 1995.

\bibitem{ChenXu2007}
{\sc Long Chen and Jinchao Xu}, {\em A posteriori error estimator by
  post-processing}, in Adaptive Computations: Theory and Algorithms, Jinchao Xu
  and Tao Tang, eds., Science Press, Beijing, 2007, pp.~34--67.

\bibitem{ChenDai2002}
{\sc Zhiming Chen and Shibin Dai}, {\em On the efficiency of adaptive finite
  element methods for elliptic problems with discontinuous coefficients}, SIAM
  J. Sci. Comput., 24 (2002), pp.~443--462 (electronic).

\bibitem{ChenZou1998}
{\sc Zhiming Chen and Jun Zou}, {\em Finite element methods and their
  convergence for elliptic and parabolic interface problems}, Numer. Math., 79
  (1998), pp.~175--202.

\bibitem{Chou2012}
{\sc So-Hsiang Chou}, {\em An immersed linear finite element method with
  interface flux capturing recovery}, Discrete Contin. Dyn. Syst. Ser. B, 17
  (2012), pp.~2343--2357.

\bibitem{Chou2015}
{\sc So-Hsiang Chou and  Champike Attanayake}, {\em Flux recovery and superconvergence
  of quadratic immersed interface finite elements}, DEC 2015.

\bibitem{ChouwKwakWee2010}
{\sc So-Hsiang Chou, Do Young Kwak, and Kye T. Wee}, {\em Optimal convergence
  analysis of an immersed interface finite element method}, Adv. Comput. Math.,
  33 (2010), pp.~149--168.

\bibitem{Ciarlet2002}
{\sc Philippe~G. Ciarlet}, {\em The finite element method for elliptic
  problems}, vol.~40 of Classics in Applied Mathematics, Society for Industrial
  and Applied Mathematics (SIAM), Philadelphia, PA, 2002.
\newblock Reprint of the 1978 original [North-Holland, Amsterdam; MR0520174 (58
  \#25001)].

\bibitem{Evans2008}
{\sc Lawrence~C. Evans}, {\em Partial differential equations}, vol.~19 of
  Graduate Studies in Mathematics, American Mathematical Society, Providence,
  RI, second~ed., 2010.

\bibitem{GongLiLi2007}
{\sc Yan Gong, Bo~Li, and Zhilin Li}, {\em Immersed-interface finite-element
  methods for elliptic interface problems with nonhomogeneous jump conditions},
  SIAM J. Numer. Anal., 46 (2007/08), pp.~472--495.

\bibitem{GuoYang2016}
{\sc Hailong Guo and Xu~Yang}, {\em Gradient recovery for elliptic interface
  problem: I. body-fitted mesh}, 2016.
\newblock arXiv:1607.05898 [math.NA].

\bibitem{GuoYang2016b}
\leavevmode\vrule height 2pt depth -1.6pt width 23pt, {\em Gradient recovery
  for elliptic interface problem: {II}. immersed finite element methods}, 2016.
\newblock arXiv:1608.000063 [math.NA].

\bibitem{GuoZhang2015}
{\sc Hailong Guo and Zhimin Zhang}, {\em Gradient recovery for the
  {C}rouzeix-{R}aviart element}, J. Sci. Comput., 64 (2015), pp.~456--476.

\bibitem{Guo2016b}
{\sc Hailong Guo, Zhimin Zhang, and Ren Zhao}, {\em Hessian recovery for finite
  element methods}, Math. Comp.,  (2016), pp.~1--22.

\bibitem{Hansbo2002}
{\sc Anita Hansbo and Peter Hansbo}, {\em An unfitted finite element method,
  based on {N}itsche's method, for elliptic interface problems}, Comput.
  Methods Appl. Mech. Engrg., 191 (2002), pp.~5537--5552.

\bibitem{HeLinLin2012}
{\sc Xiaoming He, Tao Lin, and Yanping Lin}, {\em The convergence of the
  bilinear and linear immersed finite element solutions to interface problems},
  Numer. Methods Partial Differential Equations, 28 (2012), pp.~312--330.

\bibitem{HouLiu2005}
{\sc Songming Hou and Xu-Dong Liu}, {\em A numerical method for solving
  variable coefficient elliptic equation with interfaces}, J. Comput. Phys.,
  202 (2005), pp.~411--445.

\bibitem{HouSongWangZhao2013}
{\sc Songming Hou, Peng Song, Liqun Wang, and Hongkai Zhao}, {\em A weak
  formulation for solving elliptic interface problems without body fitted
  grid}, J. Comput. Phys., 249 (2013), pp.~80--95.

\bibitem{HouWuZhang2004}
{\sc Thomas~Y. Hou, Xiao-Hui Wu, and Yu~Zhang}, {\em Removing the cell
  resonance error in the multiscale finite element method via a
  {P}etrov-{G}alerkin formulation}, Commun. Math. Sci., 2 (2004), pp.~185--205.

\bibitem{JiChenLi2014}
{\sc Haifeng Ji, Jinru Chen, and Zhilin Li}, {\em A symmetric and consistent
  immersed finite element method for interface problems}, J. Sci. Comput., 61
  (2014), pp.~533--557.

\bibitem{KwakWeeChang2010}
{\sc Do~Y. Kwak, Kye~T. Wee, and Kwang~S. Chang}, {\em An analysis of a broken
  {$P_1$}-nonconforming finite element method for interface problems}, SIAM J.
  Numer. Anal., 48 (2010), pp.~2117--2134.

\bibitem{Lakhany2000}
{\sc A.~M. Lakhany, Ivo Marek, and John Robert Whiteman}, {\em Superconvergence
  results on mildly structured triangulations}, Comput. Methods Appl. Mech.
  Engrg., 189 (2000), pp.~1--75.

\bibitem{Li1998}
{\sc Zhilin Li}, {\em The immersed interface method using a finite element
  formulation}, Appl. Numer. Math., 27 (1998), pp.~253--267.

\bibitem{LiIto2006}
{\sc Zhilin Li and Kazufumi Ito}, {\em The immersed interface method}, vol.~33
  of Frontiers in Applied Mathematics, Society for Industrial and Applied
  Mathematics (SIAM), Philadelphia, PA, 2006.
\newblock Numerical solutions of PDEs involving interfaces and irregular
  domains.

\bibitem{LiLinLin2004}
{\sc Zhinlin Li, Tao Lin, Yanping Lin, and R.~C. Rogers}, {\em An immersed finite element
  space and its approximation capability}, Numer. Methods Partial Differential
  Equations, 20 (2004), pp.~338--367.

\bibitem{LiLinWu2003}
{\sc Zhilin Li, Tao Lin, and Xiaohui Wu}, {\em New {C}artesian grid methods for
  interface problems using the finite element formulation}, Numer. Math., 96
  (2003), pp.~61--98.

\bibitem{LinYan1996}
{\sc Qun Lin and Ningning Yan}, {\em The construction and analysis of high
  efficiency finite element methods (in chinese)}, Shijiazhuang: Hebei
  University Publishers,  (1996).

\bibitem{LinYanZhou1991}
{\sc Qun Lin, NingNing Yan, and Aihui Zhou}, {\em A rectangle test for finite
  element analysis}, in Proc. System Science and System Eng.(Hong Kong), Great
  Wall Culture Publ. Co, 1991, pp.~213--216.

\bibitem{LinLinZhang2015}
{\sc Tao Lin, Yanping Lin, and Xu~Zhang}, {\em Partially penalized immersed
  finite element methods for elliptic interface problems}, SIAM J. Numer.
  Anal., 53 (2015), pp.~1121--1144.

\bibitem{NagaZhang2004}
{\sc Ahmed Naga and Zhimin Zhang}, {\em A posteriori error estimates based on
  the polynomial preserving recovery}, SIAM J. Numer. Anal., 42 (2004),
  pp.~1780--1800 (electronic).

\bibitem{NagaZhang2005}
{\sc Ahmed Naga and Zhimin Zhang}, {\em The polynomial-preserving recovery for higher
  order finite element methods in 2{D} and 3{D}}, Discrete Contin. Dyn. Syst.
  Ser. B, 5 (2005), pp.~769--798.

\bibitem{Osher2003}
{\sc Stanley Osher and Ronald Fedkiw}, {\em Level set methods and dynamic
  implicit surfaces}, vol.~153 of Applied Mathematical Sciences,
  Springer-Verlag, New York, 2003.

\bibitem{Sethian1996}
{\sc James Albert Sethian}, {\em Level set methods}, vol.~3 of Cambridge Monographs on
  Applied and Computational Mathematics, Cambridge University Press, Cambridge,
  1996.
\newblock Evolving interfaces in geometry, fluid mechanics, computer vision,
  and materials science.

\bibitem{Wahlbin1995}
{\sc Lars~B. Wahlbin}, {\em Superconvergence in {G}alerkin finite element
  methods}, vol.~1605 of Lecture Notes in Mathematics, Springer-Verlag, Berlin,
  1995.

\bibitem{WeiChenHuangZheng2014}
{\sc Huayi Wei, Long Chen, Yunqing Huang, and Bin Zheng}, {\em Adaptive mesh
  refinement and superconvergence for two-dimensional interface problems}, SIAM
  J. Sci. Comput., 36 (2014), pp.~A1478--A1499.

\bibitem{WuZhang2007}
{\sc Haijun Wu and Zhimin Zhang}, {\em Can we have superconvergent gradient
  recovery under adaptive meshes?}, SIAM J. Numer. Anal., 45 (2007),
  pp.~1701--1722.

\bibitem{Xu1982}
{\sc Jinchao Xu}, {\em Error estimates of the finite element method for the 2nd
  order elliptic equations with discontinuous coefficients}, J. Xiangtan Univ.,
  1 (1982), pp.~1--5.

\bibitem{XuZhang2004}
{\sc Jinchao Xu and Zhimin Zhang}, {\em Analysis of recovery type a posteriori
  error estimators for mildly structured grids}, Math. Comp., 73 (2004),
  pp.~1139--1152 (electronic).

\bibitem{ZhangNaga2005}
{\sc Zhimin Zhang and Ahmed Naga}, {\em A new finite element gradient recovery
  method: superconvergence property}, SIAM J. Sci. Comput., 26 (2005),
  pp.~1192--1213 (electronic).

\bibitem{ZhuLin1989}
{\sc Qiding Zhu and Qun Lin}, {\em Superconvergence Theory of the Finite
  Element Method (in Chinese)}, Hunan Science and Technique Press, Changsha,
  1989.

\bibitem{ZZ1992a}
{\sc Olek C. Zienkiewicz and Jian-Zhong Zhu}, {\em The superconvergent patch recovery
  and a posteriori error estimates. {I}. The recovery technique}, Internat.
  J. Numer. Methods Engrg., 33 (1992), pp.~1331--1364.

\bibitem{ZZ1992b}
\leavevmode\vrule height 2pt depth -1.6pt width 23pt, {\em The superconvergent
  patch recovery and a posteriori error estimates. {II}. Error estimates and
  adaptivity}, Internat. J. Numer. Methods Engrg., 33 (1992), pp.~1365--1382.

\end{thebibliography}
\end{document}